\documentclass[11pt,reqno]{amsart}
\usepackage{amsmath,amscd,amssymb,amsfonts,amsthm,courier,relsize,bm}
\usepackage{hyperref,enumerate,mathrsfs,mathtools,slashed}
\usepackage[nobysame]{amsrefs}
\usepackage{tikz-cd}
\usepackage{marginnote}
\usepackage[all,cmtip]{xy}
\usepackage[rightcaption]{sidecap}
\usepackage{graphicx}

\usepackage{xcolor}  
\hypersetup{
    colorlinks,
    linkcolor={red!50!black},
    citecolor={blue!70!black},
    urlcolor={blue!80!black}
}

\newtheorem{theorem}{Theorem}[section]

\newtheorem{proposition}[theorem]{Proposition}

\newtheorem{lemma}[theorem]{Lemma}
\theoremstyle{definition}    
\newtheorem{definition}[theorem]{Definition}
\theoremstyle{boldremark}

\newtheorem{remark}[theorem]{Remark}

\theoremstyle{boldexample}
\newtheorem{example}[theorem]{Example}

\newcommand{\pair}[2]{\langle #1, #2 \rangle}
\newcommand{\ignore}[1]{}

\newcommand{\ol}[1]{\overline{#1}}

\newcommand{\sm}[1]{\mathsmaller{#1}}

\def\d{\ensuremath{\mathrm{d}}}

\def\ad{\ensuremath{\textnormal{ad}}}
\def\g{\ensuremath{\mathfrak{g}}}

\def\t{\ensuremath{\mathfrak{t}}}

\def\h{\ensuremath{\mathfrak{h}}}

\def\A{\ensuremath{\mathcal{A}}}
\def\B{\ensuremath{\mathcal{B}}}
\def\C{\ensuremath{\mathcal{C}}}

\def\G{\ensuremath{\mathcal{G}}}

\def\L{\ensuremath{\mathcal{L}}}

\def\R{\ensuremath{\mathcal{R}}}

\def\U{\ensuremath{\mathcal{U}}}

\def\bC{\ensuremath{\mathbb{C}}}
\def\bR{\ensuremath{\mathbb{R}}}
\def\bZ{\ensuremath{\mathbb{Z}}}

\def\Hom{\ensuremath{\textnormal{Hom}}}

\def\ker{\ensuremath{\textnormal{ker}}}

\def\mon{\ensuremath{\textnormal{Mon}}}
\def\res{\ensuremath{\textnormal{res}}}

\def\tr{\ensuremath{\textnormal{tr}}}
\def\pr{\ensuremath{\textnormal{pr}}}
\def\dim{\ensuremath{\textnormal{dim}}}

\def\Cl{\ensuremath{\textnormal{Cl}}}

\def\sm{\ensuremath{\textnormal{sm}}}
\def\bk{\ensuremath{{}^{b^k}}}
\def\spinc{\ensuremath{\textnormal{Spin}}^c}

\def\Ch{\ensuremath{\textnormal{Ch}}}
\def\Td{\ensuremath{\textnormal{Td}}}

\title{$b^k$-Symplectic manifolds and $[Q,R]=0$}
\author{Ahmad Reza Haj Saeedi Sadegh }
\date{}
\email{Ahmadreza.Hajsaeedisadegh@dartmouth.edu}
\address{6188 Kemeny Hall\\ 
         Dartmouth College\\
         Hanover, NH 03755}
\subjclass[2020]{53D50, 53D17, 53D20, 58J20, 58H05, 19K56, 46L80, 81S10}
\keywords{$b^k$-symplectic manifold, log-symplectic manifold, Lie algebroid,
Lie groupoid, geometric quantization, symplectic reduction, index theory,
modular weight}

\begin{document}

\maketitle

\begin{abstract}

   We study the geometric quantization of $b^k$-symplectic manifolds using the
integrability of Lie algebroids. Using a groupoid index, we define a
quantization for $b^k$-symplectic manifolds whose singular locus is a normal
crossing divisor and which carry a Hamiltonian action of a compact connected
Lie group, generalizing Guillemin--Miranda--Weitsman in a few directions.
Firstly, we show this quantization is the index of a $\spinc$-Dirac operator, answering a
question of theirs. In particular, it is a finite-dimensional virtual
representation for every $k$, whereas their formal quantization is infinite
dimensional when the modular degrees are even. Secondly, our symplectic form can have singularities along hypersurfaces which can have normal crossings. 
Finally, we prove that 
quantization commutes with reduction for the Hamiltonian action of a possibly non-abelian compact connected Lie group, when the modular degrees are odd. In the case when the modular degrees are not odd, we give an example when $[Q,R]=0$ fails.

\end{abstract}

\section{Introduction}

In this article, we study the geometric quantization of two classes of Poisson manifolds arising from \emph{symplectic algebroids}. 
We utilize the integrability of these algebroids to desingularize the symplectic structure and to compute index formulas for Dirac operators on Lie groupoids; this technique was developed in \cite{sadegh2024fixed}.

Let $\mathcal{A}\to M$ be a Lie algebroid with anchor $\rho:\mathcal{A}\to TM$. Let $\omega\in \Gamma(\Lambda^2\A^*)$ be a symplectic form, that is

\begin{enumerate}
     \item  $\omega $ is nondegenerate i.e. the bundle morphism $\omega:\A\to \A^*$ is an isomorphism.

    \item $\omega$ is closed with respect to the de Rham differential of the algebroid.
\end{enumerate}
The pair $(\mathcal{A},\omega)$ is called a \emph{symplectic algebroid}. Such a pair gives rise to a Poisson structure on $M$. One can define the bundle map $\alpha:T^*M\to TM$

\[\xymatrix{\mathcal{A}\ar[r]^{\rho}&TM\\
\mathcal{A}^*\ar[u]^{\omega^{-1}}&  \ar[l]^{\rho^*} T^*M \ar[u]_{\alpha}}.\]
It turns out that the morphism $\alpha$ is given by pairing with a Poisson bivector. Recently, \cite{cren2026symplectic} proved the deformation quantization of Kontsevich \cite{kontsevich2003deformation} of Poisson manifolds arising from symplectic algebroids. In this article, we study the geometric quantization of two classes of such Poisson structures known as Log-symplectic manifolds and $b^k$-symplectic manifolds. The geometric quantization of Log-symplectic manifolds has been studied by \cite{guillemin2018geometric,braverman2019geometric,lin2022log}, and for $b^k$-symplectic manifolds by \cite{guillemin2021geometric}. We generalize these results of the latter in two directions. First, we let
the $b^k$-symplectic form have singularities along hypersurfaces that can have normal crossing
intersections. Second, we let the Hamiltonian action be given by a general compact connected
Lie group. We then show that our quantization commutes with the reduction.

\subsection{E-manifolds}
Let $M$ be a manifold and let $\tau$ be a subsheaf of vector fields on $M$ with the following properties:
\begin{enumerate}
    \item $\tau$ is closed under the Lie bracket of vector fields. 

    \item $\tau$ is a locally free sheaf of constant rank.
\end{enumerate}
From the second property and the Serre-Swan theorem, it follows that there exists a vector bundle $\A\to M$ whose sections identify with sections of the sheaf $\tau$, and also there exists a canonical bundle map $\rho:\A\to TM$. The first property gives a Lie bracket on $\Gamma(\A)$. In total, we obtain a Lie algebroid:
\[\xymatrix{\A\ar[r]^{\rho}\ar[rd]&TM\ar[d]\\  &M}.\]
The pair $(M,\tau)$ is introduced in \cite{miranda2021geometry} and is called an $ E$-manifold, and the corresponding algebroid $\A$ is called an $ E$-tangent bundle. One has the following
\begin{proposition}\label{prop:E-tangentintegrable}
    $E$-tangent bundles are always integrable.
\end{proposition}

\begin{proof}[\textbf{Proof of Proposition \ref{prop:E-tangentintegrable}}]
The proof follows a result of Debord on almost injective algebroids. Let $\tau$ be a locally free subsheaf of vector fields that is of constant rank and closed under the Lie bracket, and let $\A$ be the corresponding $E$-tangent bundle. In particular $\tau$ is a (possibly singular) foliation. By basic results on singular foliations (e.g. \cite{laurent2024invitation}), there exists an open and dense subset of $M$ over which the foliation is regular, and the leaves are of maximal dimension. Over this open set, the algebroid $\A$ has an injective anchor. Hence, by \cite{debord2000groupoi,debord2001holonomy}, the algebroid is integrable.
\end{proof}
\begin{remark}
    This proof was suggested to the author by Christian Blohmann and Clement Cren. 
\end{remark}

There are important examples of $E$-manifolds (c.f. \cite{miranda2021geometry}):
\begin{example}
    \begin{itemize}
        \item \textbf{$b$-tangent bundle and $b^k$-tangent bundle}: Let $Z\subset M$ be a (possibly disconnected) hypersurface. Then the sheaf of vector fields on $M$ that are tangent to $Z$ gives an E-manifold known as \emph{$b$-manifold}. By adding $k$-jet data along $Z$, one can define a  \emph{$b^k$-manifold}.

        \item \textbf{log-tangent bundle} Let $Z$ be an immersed submanifold of $M$ with self-intersections of normal crossing type. Then one can still define the sheaf of vector fields on $M$ tangent to $Z$. This E-manifold is called a \emph{log-manifold} or \emph{c-manifold}. 
    \end{itemize}
\end{example}

When an $E$-tangent bundle $\A$ is paired with a symplectic form $\omega$, we call the pair $(\A,\omega)$ an \emph{$E$-symplectic manifold}. Similarly, we use the terms $b^k$-symplectic and log-symplectic manifolds when a symplectic form is given.

\subsection{Prequantizability of E-symplectic manifolds}
Let $(\mathcal{A},\omega)$ be a symplectic algebroid over a closed manifold $M$. In order to create prequantization data, we need a line bundle $\L\to M$ whose first Chern class is prescribed by the symplectic form $\omega$. In general, we do not yet have a general prescription. However, we give a prescription for log- and $b^k$-symplectic manifolds:

\begin{definition}
    Let $(\A,\omega)$ be a log- or $b^k$-symplectic manifold. We call $(\A,\omega)$ prequantizable if the smooth part of the cohomology class $\sm[\omega]\in H^{2}(M)$ is integral. In this case, there exists a prequantum line bundle $L$ for which $c_1(\L)=\sm[\omega]$. 
\end{definition}

\begin{remark}
    This prequantization condition coincides with the prequantization condition in \cite{lin2022log} for log-symplectic manifolds and is more relaxed than that of \cite{guillemin2021geometric} for $b^k$-symplectic manifolds. 
\end{remark}

\subsection{Geometric quantization}\label{ss:geometricquantizationmethod}
We give our geometric quantization method here. Let $(\A,\omega)$ be a symplectic algebroid for which the prequantum line bundle $\L\to M$ is given (this assumption is clear for log- and $b^k$-symplectic settings). Let $J$ be an $\omega$-compatible complex structure on the fibers of $\A$. So we have decompositions $\A\otimes\bC=\A^{(1,0)}\oplus\A^{(0,1)}$ corresponding dual spaces $\A^{*,(1,0)}$ and $\A^{*,(0,1)}$. We form the spinor vector bundle $S=S^+\oplus S^-$ with 
\[S^+:=\bigoplus_{r \ \text{even}}\Lambda^r\A^{*,(0,1)}\otimes \L, \ \ \ \ S^-:=\bigoplus_{r \ \text{odd}}\Lambda^r\A^{*,(0,1)}\otimes \L\]
which is $\Cl(\A^*)$-module. Here $\Cl(\A^*)$ is the associated complex Clifford algebra bundle given with respect to the metric $g(v,w):=\omega(v,Jw)$. Similar to the usual Clifford modules, one can define a canonical differential operator, \emph{algebroid-Dirac operator}, on $\Gamma(S)$:
\[D=\begin{pmatrix}
0 & D^-\\
D^+ & 0
\end{pmatrix}:\Gamma(S^+)\oplus\Gamma(S^-)\to\Gamma(S^+)\oplus\Gamma(S^-).\]
This operator is obtain by composing a compatible \emph{algebroid-connection} $\nabla$ with the Clifford multiplication. The algebroid connection on $S$ is obtained by tensoring the algebroid connection on the $\Lambda^*\A^{*,(0,1)}$ induced from the Levi-Civita connection on the algebroid $\A$, with an algebroid connection on $\L$ whose curvature has the same cohomology class as the $\sm[\omega]$.

Like Dirac operators, this is a first-order, essentially self-adjoint differential operator, but it is elliptic if and only if the anchor of $\A$ is surjective. So in the $E$-symplectic case, this operator is not elliptic unless the sheaf is the set of all vector fields. Despite the lack of ellipticity, one still can define an index for such operators. For more details on the algebroid-Dirac operators, see \cite{sadegh2024fixed}.

To define this index, we take a groupoid $\G\rightrightarrows M$ integrating the algebroid $\A$. The algebroid-Dirac operator gives rise to a family of source-wise Dirac operators on the groupoid due to the fiberwise ellipticity; it gives an index class in the K-theory group $[D]\in K_0(\Psi^{-\infty}(\G))$ where $\Psi^{-\infty}(\G)$ is the algebra of smoothing operators on the groupoid. In the log- or $b^k$-case, one has a canonical trace $\tr:\Psi^{-\infty}(\G)\to \bC$; by pairing with the trace we obtain the geometric quantization:
\[Q(\A,\omega):=\tr([D]).\]

 It is not clear in the outset that the quantization above is an integer (see Remark \ref{rem:nonintegralquantization}). In the case of the log-tangent bundle for a defined normal crossing divisor (see Section \ref{sec:logsymplectic}), or in the case of $b^k$-tangent bundle, we will show the quantization is the index of an honest (elliptic) $\spinc$-Dirac operator.
\medskip
\begin{theorem}\label{thm:logmanifoldquantizationformula}
    Let $(M,Z)$ be a generalized normal crossing divisor with the associated log-tangent bundle $\A=T_ZM$. If $\omega$ is a prequantized symplectic form, then quantization $Q(\A,\omega)$ coincides with the geometric quantization formula in \cite{lin2022log}. Indeed
    \[Q(\A,\omega)=\text{pv}_Z\int_M\Td^{\A}(\A)e^{\sm[\omega]}.\]
    Here $\text{pv}_Z\int_M$ is the principal value integral of singular forms, and $\Td^{\A}$ is the algebroid-Todd class. 
\end{theorem}

\begin{remark}
    In the case of defined normal crossing divisors, Theorem \ref{thm:logmanifoldquantizationformula} recovers the quantization formula in \cite{lin2022log}.
\end{remark}

The main result of this article is the following: 
\begin{theorem}\label{thm:b^kmanifoldquantizationformula}
    Let $(M,Z)$ be a defined normal crossing divisor. Let $\A=\bk T_ZM$ be the associated $b^k$-tangent bundle. If $\omega$ is a prequantized symplectic form on $\A$, then the geometric quantization is given by the following formula
    \[Q(\A,\omega)=\text{fp}_Z\int_M\Td^{\A}(\A)e^{\sm[\omega]},\]
    where $\text{fp}_Z\int_M$ denotes the Liouville volume of $b^k$-forms. We will show the quantization is the index of an honest $\spinc$-Dirac operator, hence an integer.
    \end{theorem}

    \begin{theorem}\label{thm:bk[Q,R]=0intro}
    When $(\A,\omega, \mu)$ is a prequantized $G$-Hamiltonian compact $b^k$-symplectic manifold one has the formula
    \[Q^G(\bk T_ZM,\omega,\mu)(g)=\textup{fp}_{Z^g}\int_{M^g}\frac{\Td^{\A}(\bk T_{Z^g}M^g)e^{\sm[\omega]-\bar{\mu}}}{\textup{ch}_g(\lambda_{-1}\nu^{0,1})}.\]
    If one assumes:
    \begin{enumerate}
        \item  Assume  $0$ is regular value for $\mu$, and $G$ acts freely on $\mu^{-1}(0)$;

        \item For every stratum $N=W_1\cap\cdots\cap W_p$ where $W_1,\cdots, W_p$ are any distinct components of $Z_{\neq0}$, the non-negative linear combination of $\{c_{W_1},\cdots, c_{W_p}\}$ is strongly convex;

        \item $(M,Z,x_i,\mu)$ has odd leading type.
    \end{enumerate}    
    Then the {quantization commutes with reduction} in the sense of \cite{guillemin1982geometric,meinrenken1999singular}:
    \[(Q^G(\bk T_ZM,\omega,\mu))^G=Q(\bk T_{Z_0}M_0,\omega_0).\]
    
\end{theorem}

\medskip
\begin{remark}
    Theorems \ref{thm:b^kmanifoldquantizationformula} and \ref{thm:bk[Q,R]=0intro} are significant in three ways. First, it provides a positive answer to the conjecture by Guillemin, Miranda and Weitsman \cite{guillemin2021geometric} that, when it is finite, the quantization is an index of a Fredholm operator. Second, the quantization formula by Guillemin, Miranda and Weitsman is not finite when $k$ is even; however, our quantization is finite, but it does not satisfy the $[Q,R]=0$ condition (see Example \ref{ex:QRfails}); the failure is due to the violation of Condition (3) in Theorem \ref{thm:bk[Q,R]=0intro}. Third, we generalize the quantization of \cite{guillemin2021geometric} from the Hamiltonian action of a torus and disjoint hypersurfaces to any compact connected Lie group and normal crossing divisors.
\end{remark}

\begin{remark}
    Conditions (2) and (3) in Theorem \ref{thm:bk[Q,R]=0intro} are crucial in the $[Q,R]=0$ statement. When $k=1$, Condition (3) is vacuous, while Condition (2) is equivalent to the properness of the moment map (see \cite[Lemma 4.8]{lin2022log}). When the hypersurfaces have no crossings, again, Condition (2) is vacuous.  
\end{remark}

\medskip
\subsection{Structure of the article}
Section \ref{sec:logsymplectic} reviews normal crossing divisors, log-tangent bundles, and log-symplectic manifolds. Subsection \ref{ss:logtangent}, we define the log-tangent bundle for the immersed hypersurfaces and analyze its vector bundle structure; then we recover the algebroid structure in Subsection \ref{ss:liealgT_ZM}. In Subsection \ref{ss:principalvalue}, we discuss integration of log-forms, and in the subsequent Subsection \ref{ss:cohomologyT_ZM} we compute algebroid cohomology of the log-tangent bundle. We discuss the geometric quantization of log-symplectic manifolds in Subsection \ref{ss:geometricquantizationlogsymlectic}, and show our formula recovers the quantization of \cite{lin2022log}.

Section \ref{sec:bkmanifold} covers the main results of the article. 
We generalize the notion of $b^k$-tangent bundles to the case with defined hypersurfaces that have normal crossings. We generalize the Liouville volume of $b^k$-forms by Scott \cite{scott2013geometry} to our general setting in Subsection \ref{ss:liouvillevolume}, and in Subsection \ref{ss:cohomologybktangentbundle} we compute the algebroid cohomology of the $b^k$-tangent bundle. Subsection \ref{ss:geomtricquantizationb^ksymplectic} gives a detailed construction of the geometric quantization of $b^k$-symplectic manifolds, and Subsection \ref{ss:bkHamiltonianquantization} introduces the Hamiltonian setting and its quantization. In Subsection \ref{ss:quantizationFredhol}, we prove a conjecture by Guillemin, Miranda, and Weitsman \cite{guillemin2021geometric}, by showing that the geometric quantization equals the index of an $\spinc$-Dirac operator. The ``quantization commutes with reduction'' ($[Q,R]=0$) is stated in Subsection \ref{ss:[Q,R]=0}. Subsection \ref{ss:counterexample} is dedicated to a counterexample when Condition (3) in our theorem fails.

The proof of $[Q,R]=0$ is given in Section \ref{sec:[Q,R]proof}.

\subsection*{Background material}
In this paper, we rely on the basic knowledge of the differential topology of Lie algebroids. This includes the algebroid cohomology, algebroid connections, and Chern-Weil theory of algebroids. See \cite[Section 2]{sadegh2024fixed} for more details.

\subsection*{Acknowledgements} The author wants to thank Christian Blohmann, Ed McDonald, Yiannis Loizides, and Nigel Higson for the helpful discussions and suggestions. The author wants to thank Clement Cren and Erfan Rezaei Gharehbolagh for great discussions on Groupoids and K-theory. The author wants to thank the Max Planck Institute for Mathematics in Bonn for the scholarship and hosting during the summers of 2025 and 2026, where most of the ideas of this article brewed.

\subsection*{Declaration of generative AI use in the writing process} During the preparation of this work, the author used Claude (Anthropic, model Claude Opus 5) to write code for generating the figures from the author's own data, and for grammatical corrections and proofreading. The author reviewed and edited all output and takes full responsibility for the content of the publication. The core ideas, arguments, and conclusions are entirely the author's own.

\section{Log-symplectic Manifolds}\label{sec:logsymplectic}

Let $M$ be a manifold of dimension $n$ and let $Z$ be a not-necessarily-connected closed manifold of dimension $n-1$ with an immersion $i:Z\to M$. We say the immersion is \emph{self-transverse} if for all distinct points $z_1,\cdots,z_k\in Z$ that map to the same point $m\in M$ under the immersion, the tangent spaces $\{i_*(T_{z_i}Z)\}_{i=1}^k$ are in general position as subspaces of $T_mM$, i.e., they intersect like coordinate hyperplanes of $\mathbb{R}^n$. The data $(M,Z)$ above is referred to as a $c$-manifold in \cite{miranda2021geometry}, but we refer to this as a \emph{generalized normal crossing divisor} to be more in line with the terminology in \cite{lin2022log}. 

By \emph{defined normal crossing divisors} $(M,Z,x_i)$, we additionally mean that one can write $Z$ as a disjoint union of open submanifolds $Z_1\sqcup \cdots\sqcup Z_p$ such that $i:Z_j\to M$ is an embedding and a defining function $x_j$ for each $Z_j\simeq i(Z_j)$. We will use the defined normal crossing divisors in Section \ref{sec:bkmanifold}.

\graphicspath{ {images/} }


\begin{figure}[ht]
\centering
\begin{tikzpicture}[
    x=1cm, y=1cm, font=\footnotesize,
    tor/.style    = {draw=black!60, line width=0.5pt},
    rib/.style    = {draw=black!16, line width=0.35pt},
    eight/.style  = {draw=blue!55!black, line width=1.2pt},
    eighth/.style = {draw=blue!55!black, line width=0.6pt, dash pattern=on 2pt off 2pt},
    ax/.style     = {draw=black!70, line width=0.5pt, -{Stealth[length=4pt]}},
    lev/.style    = {draw=black!45, line width=0.4pt, dash pattern=on 1.6pt off 1.6pt},
    dp/.style     = {circle, fill=black, draw=white, line width=0.5pt, inner sep=1.5pt},
    cp/.style     = {circle, draw=black!70, fill=white, line width=0.45pt, inner sep=1pt},
    lbl/.style    = {inner sep=1.5pt},
    cal/.style    = {draw=black!55, line width=0.4pt, -{Stealth[length=3.5pt]}},
    chart/.style  = {draw=black!45, line width=0.5pt, fill=black!3},
    sheet/.style  = {draw=black!75, line width=1.05pt}
  ]
 
\def\R{2.0}\def\rr{0.88}\def\al{42}\def\aa{-1.12}
\newcommand{\Sx}[2]{(\R+\rr*cos(#2))*cos(#1)}
\newcommand{\Sy}[2]{\rr*sin(#2)*sin(\al)+(\R+\rr*cos(#2))*sin(#1)*cos(\al)}
\newcommand{\vsil}[1]{(atan(sin(#1)*tan(\al)))}
\newcommand{\sarg}[1]{max(-1,min(1,\aa/(\R+\rr*cos(#1))))}
\newcommand{\uA}[1]{(asin(\sarg{#1}))}
\newcommand{\uB}[1]{(-180-asin(\sarg{#1}))}
 
\begin{scope}
 
  \draw[lev] (-4.08,-0.8323) -- (-0.12,-0.8323);
 
  \foreach \uu in {0,30,60,90,120,150,180,210,240,270,300,330}{
    \draw[rib] plot[domain=0:1, samples=42, variable=\s]
      ({\Sx{\uu}{\vsil{\uu}-180+180*\s}}, {\Sy{\uu}{\vsil{\uu}-180+180*\s}});
  }
 
  \draw[tor] plot[domain=0:360, samples=161, variable=\u]
      ({\Sx{\u}{\vsil{\u}}}, {\Sy{\u}{\vsil{\u}}});
  \draw[tor] plot[domain=0:360, samples=161, variable=\u]
      ({\Sx{\u}{\vsil{\u}+180}}, {\Sy{\u}{\vsil{\u}+180}});
 
  \draw[eighth] plot[domain=-19.6:142.3, samples=110, variable=\v]
      ({\Sx{\uA{\v}}{\v}}, {\Sy{\uA{\v}}{\v}});
  \draw[eight]  plot[domain=142.3:340.4, samples=130, variable=\v]
      ({\Sx{\uA{\v}}{\v}}, {\Sy{\uA{\v}}{\v}});
  \draw[eighth] plot[domain=-19.6:142.3, samples=110, variable=\v]
      ({\Sx{\uB{\v}}{\v}}, {\Sy{\uB{\v}}{\v}});
  \draw[eight]  plot[domain=142.3:340.4, samples=130, variable=\v]
      ({\Sx{\uB{\v}}{\v}}, {\Sy{\uB{\v}}{\v}});
 
  \node[dp] (nd) at (0,-0.8323) {};
  \node[cp] at (0,2.1403)  {};
  \node[cp] at (0,0.8323)  {};
  \node[cp] at (0,-2.1403) {};
 
  \begin{scope}[shift={(-4.15,0)}]
    \draw[ax] (0,-2.62) -- (0,2.64);
    \node[lbl,above] at (0,2.64) {$f$};
    \foreach \h/\t in {-2.1403/{\min f}, -0.8323/{a}, 0.8323/{b}, 2.1403/{\max f}}
      { \draw (-0.07,\h) -- (0.07,\h);  \node[lbl,left] at (-0.12,\h) {$\t$}; }
  \end{scope}

  \node[lbl, text=blue!55!black] (zl) at (2.86,-1.72) {$i(Z)=f^{-1}(a)$};
  \draw[cal] (2.68,-1.58) -- (2.52,-1.16);
 
  \node[lbl,align=center] (w2) at (-1.86,-2.34) {$W_2$: the saddle};
  \draw[cal] (-1.42,-2.16) -- (-0.14,-0.92);
 
  \node[lbl] at (2.10,1.86) {$M$};
\end{scope}
 
\begin{scope}[shift={(6.30,-0.15)}]
 
  \draw[sheet] (1.25,2) circle (0.75cm) ;
  \node[lbl,left]  at (1.15,2.85) {$z_1$};
  \node[dp] at (1.25,2.75) {};
  \node[lbl,left]  at (1.15,1.15) {$z_2$};
  \node[dp] at (1.25,1.25) {};
  \node[lbl,left]  at (2.45,2) {$Z$};
 
  \draw[cal] (1.30,1.05) -- (1.30,0.55);
  \node[lbl,right] at (1.34,0.80) {$i$};
 
  \draw[chart] (-0.05,-1.85) rectangle (2.65,0.40);
  \draw[sheet] (0.20,-1.62) -- (2.42,0.20);
  \draw[sheet] (0.20,0.20)  -- (2.42,-1.62);
  \node[dp] at (1.31,-0.71) {};
  \node[lbl,left]  at (1.18,-0.88) {$m$};
  \node[lbl,right] at (2.72,0.20)  {};
  \node[lbl,right] at (2.72,-1.62) {};
  \node[lbl,right] at (-0.02,-2.06) {$M$};
 
  \node[lbl,align=left] at (1.30,-2.72)
     {$\dim\ker\rho_m=2$ on $W_2$,\ $1$ on $W_1$,\ $0$ off $i(Z)$};
\end{scope}
 
\end{tikzpicture}
 
\caption{Consider a torus standing so its hole faces the reader, with the standard height function $f$ which is Morse with four critical points of heights $\min f<a<b<\max f$, as marked on the axis at the left.
 For a regular value the level set is embedded, a single circle or a
pair of circles. At the saddle value $a$, however, the two circles collide and
$f^{-1}(a)$ becomes an \emph{infinity sign} which is the image of an immersion of a circle with one double point. So $W_2$ is the single point $i(M_2)$ and $M_2$ consists of the two
ordered pairs of preimages of the saddle, while $W_1=i(Z)\setminus i(M_2)$ is
the remaining one-dimensional part. Since $\dim M=2$, at most two branches can
be in general position at a point, so $M_k=\emptyset$ for $k>2$. }
\label{fig:immersedZ}
\end{figure}
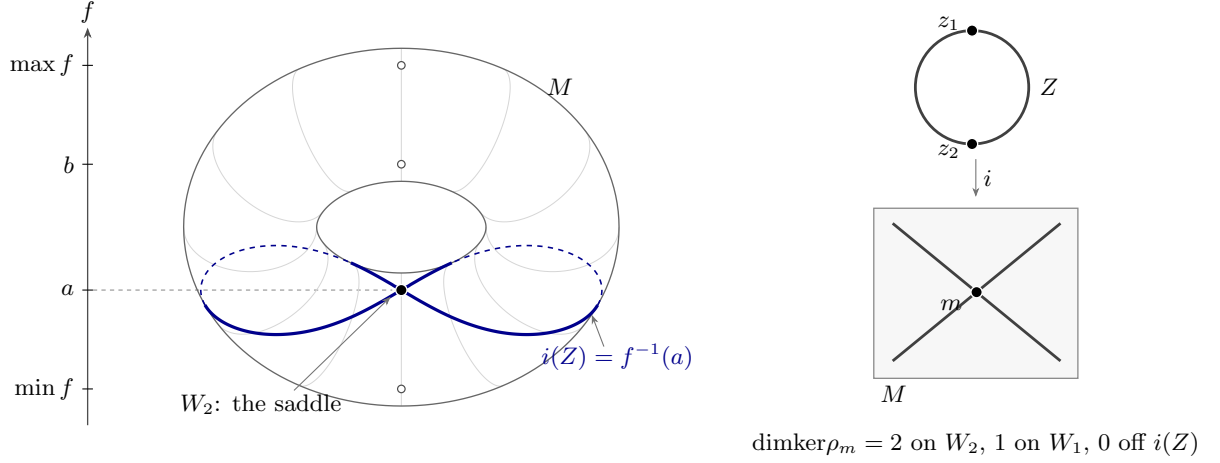


\subsection{log-tangent bundle}\label{ss:logtangent}

For an open subset $U$ of $M$, let $\Gamma_Z(U)\subset \Gamma(TU)$ be the collection of vector fields $X$ on $U$ that are ``tangent to $Z$", that is, for all $z\in Z$ with $i(z)\in U$, we have $X_{i(z)}\in i_*(T_{z}Z)$. This defines a sheaf $\Gamma_Z$ on $M$ that turns out to be locally free and of constant rank $n$ (see \cite[Section 3]{miranda2021geometry}). Hence $\Gamma_Z(M)$ identifies with the space of global sections of a vector bundle of rank $n$:

\begin{definition}
Denote by $T_ZM\to M$ the vector bundle associated with the sheaf $\Gamma_Z$. We call this vector bundle the log-tangent bundle associated with $(M,Z)$.
\end{definition}

The vector bundle $T_ZM$ and the tangent bundle $TM$ are very similar. Indeed, there is a canonical bundle map 
\[\rho:T_ZM\to TM\]
which is an isomorphism for fibers over $m\in M\backslash i(Z)$. We can say more:

\begin{proposition}\label{prop:kernelofanchor}
If $m\in M$ is the image of exactly 
$k$ points of $Z$, then 
the kernel of $\rho_m:T_ZM|_m\to TM|_m$ is $k$-dimensional. 
\end{proposition}
When a point $m\in M$ is the image of $k$ points of $Z$, then $m$ belongs to an $(n-k)$-dimensional submanifold of $M$ formed by the ``self-intersections'' of $i(Z)$; 
Indeed, put $M_0:=M$ and for $k>0$ define 
\[{M_k}:=\{(z_1,\cdots,z_k)\in Z^k| z_i \ \textup{distinct}, i(z_i)=i(z_j) \}.\]
 Note that $M_k$ is a smooth manifold of dimension $n-k$. For $k>l$, by projecting onto the first $l$ components, we obtain self-transverse immersion 
\[M_k\to M_l \]
When $l=0$, denote $i:M_k\to M$; then the difference $W_k:=i(M_k)-i(M_{k+1})$ is an embedded $(n-k)$-dimensional submanifold consisting of those points of $M$ that are images of exactly $k$ distinct points of $Z$.
\subsection{Lie algebroid structure}\label{ss:liealgT_ZM} Note that the sheaf $\Gamma_Z$ is closed under the Lie bracket of the vector field. This gives a Lie algebroid structure on the log-tangent bundle:

\[\xymatrix{T_ZM\ar[r]^{\rho}\ar[rd]&TM\ar[d]\\
{}& M }\]
for which the canonical map $\rho:T_ZM\to TM$ is the anchor.

\medskip
\subsection{Principal value integral}\label{ss:principalvalue}
Let $\alpha\in\Gamma(\Lambda^nT_Z^*M)$ be a top-degree form. Locally, near the orbit $W_k$, the form $\alpha$
is of the form
\[\alpha=f\frac{dx_1}{x_1}\cdots \frac{dx_k}{x_k}dx_{k+1}\cdots dx_{n}+\beta.\]
Here $(x_1,\cdots,x_n)$ is a local coordinate near $W_k$, with $W_k=\{x_1=\cdots=x_k=0\}$; $f$ is a smooth function, and $\beta$ is a smooth form. Note that the integral 
\[\int_M \alpha\]
can blow up due to the singularity. However, we can extract the \emph{finite part} of this integral as follows. Define the principal value integral by

\[\textup{pv}_Z\int_M\alpha=\lim_{\epsilon\to 0}\int_{M_{\epsilon}}\alpha\]
where $M_{\epsilon}$ is the complement of an $\epsilon$-neighborhood of image of $Z$ in $M$ with respect to any Riemannian metric. See \cite{miranda2021geometry} for more details.

\medskip
\subsection{Algebroid cohomology}\label{ss:cohomologyT_ZM} Let $\A\to M$ is an algebroid with anchor $\rho:\A\to TM$. One defines the algebroid de Rham differential $d_{\A}:\Gamma(\Lambda^r\A^*)\to\Gamma(\Lambda^{r+1}\A^*)$ by
\[d_{\A}\alpha(X_0,\cdots,X_{r})=\sum_{i=0}^r(-1)^i\rho(X_i).\alpha(X_0,\cdots,X_{i-1},X_{i+1},\cdots,X_r)\]
\[+\sum_{i<j}(-1)^{i+j}\alpha([X_i,X_j],X_0,\cdots,X_{i-1},X_{i+1},\cdots,X_{j-1},X_{j+1},\cdots,X_r).\]
One has $d_{\A}^2=0$, and hence one can define the algebroid cohomology groups $H^i(\A)$, for $i=0,\cdots,n$, where $n$ is the rank of $\A$.

For the case of $\A=T_ZM$, the cohomology of the algebroid is closely related to the cohomology of the base:

\begin{proposition}\label{prop:cohomologylogtangent}
    There is a canonical isomorphism
    \[H^r(T_ZM)\simeq H^r(M)\oplus \hat{H}^{r-1}(Z_1)\oplus \hat{H}^{r-2}(Z_2)\oplus\cdots; \]
    Here $\hat{H}^i(Z_j)$ denotes the ``compatible cohomology'' of the manifolds $Z_j$ with respect to the free action of the symmetric group $S_j$. We will not define these groups as we will not use them here; see \cite{miranda2021geometry} for more details.
\end{proposition}

We will denote by $\sm:H^r(T_ZM)\to H^r(M)$ the canonical projection, following \cite{miranda2021geometry}. We call $\sm[\alpha]$ the smooth part of the class $[\alpha]$. One way to describe the smooth part map is as follows. Assume $M$ is orientable; then

\[\sm:H^r(T_ZM)\to (H^{n-r}(M))^*\buildrel{\textup{Poincar\'e}}\over\simeq H^r(M)\]
\[[\alpha]\mapsto [[\mu]\mapsto \textup{pv}_Z\int_M\alpha\wedge\mu].\]

\medskip
\subsection{Geometric quantization of the log-symplectic manifolds}\label{ss:geometricquantizationlogsymlectic}

Assume $(M,Z)$ is a \emph{generalized normal crossing divisor} with $M,Z$ closed manifolds. Let $\omega$ be a symplectic form on the log-tangent bundle $T_ZM$ which we assume is prequantizable, that is $\sm[\omega]$ is an integral class. Let $\L\to M$ be a line bundle with $c_1(\L)=\sm[\omega]$. Choosing a compatible complex structure $J$ on $\A$, as in Subsection \ref{ss:geometricquantizationmethod}, we may define an algebroid-Dirac operator $D$ on an $\Cl(\A^*)$-module $S$. 

Let $r,s:\G\rightrightarrows M$ be any integration of the algebroid $\A$ which exists by Theorem \ref{prop:E-tangentintegrable}. This operator gives an index class $[D]\in K_0(\Psi^{-\infty}(\G))$. In \cite{sadegh2024fixed}, a canonical trace $\tr:\Psi^{-\infty}(\G)\to\mathbb{C}$ introduced, which gives our quantization:
\[Q(T_ZM,\omega):=\tr[D].\]

This trace is given as follows. We have the identification $\Psi^{-\infty}(\G)\simeq \Gamma(\delta\Lambda^{1/2})$ where 
\[\delta\Lambda^{1/2}:=r^*\Lambda^{1/2}\otimes s^*\Lambda^{1/2}\]
and $\Lambda$ is the density bundle $\Lambda=|\det T_Z^*M|\otimes\bC.$ Now the trace is given as the composition
\[\tr:\Gamma(\delta\Lambda^{1/2})\to\Gamma(\Lambda)\to \bC\]
where the first map is the restriction to the unit space of the groupoid and the second map is the Liouville volume introduced above.

\begin{remark}
    In \cite{sadegh2024fixed}, one assumes the components of $Z$ are embedded. However, the results extend to our case of the \emph{generalized normal crossing divisors} with minimal adjustments in the arguments.
\end{remark}

By the index formula \cite[Theorem 6.10 \& Corollary 8.6]{sadegh2024fixed} that this quantization computes as 
\[Q(T_ZM,\omega)=\text{pv}_Z\int_M{\Td^{\A}}(T_ZM)e^{\sm[\omega]},\]
where $\Td^{\A}$ denotes the algebroid version of the Todd class. This proves Theorem \ref{thm:logmanifoldquantizationformula}.

\begin{remark}
    Assume $(T_ZM,\omega)$ is a log-symplectic manifold, where $i:Z\to M$ restricted to components of $Z$ gives an embedding whose image is given by a defining function. This is the setting in \cite{lin2022log}. Our geometric quantization recovers their formula as follows. 

    In this case, one has $T_ZM\oplus\bR\simeq TM\oplus \bR$ (\cite[Theorem 2.8 ]{lin2022log}); one obtains a canonical $\spinc$-structure on $M$, and hence a Dirac operator $\slashed{D}$. 
    By twisting with the prequantum line bundle $\L$, one obtain the geometric quantization of \cite{lin2022log}:
    \[\textup{Ind}(\slashed{D}^{\L})=\int_M \hat{A}^{\A}(M)\Ch(\textup{det}^{1/2}_{\bC}(TM\oplus \bC))\Ch(\L)\]
    \[=\int_M \hat{A}^{\A}(T_ZM)\Ch(\textup{det}^{1/2}_{\bC}(T_ZM))\Ch(\L)\]
    Here one can show that $\hat{A}^{\A}(T_ZM)$ is the same as the algebroid version A-hat of $T_ZM$. Thus we obtain
    \[\textup{Ind}(\slashed{D})=\int_M\hat\Td(T_ZM)\Ch(L)\]
    which is exactly our formula.
\end{remark}
\medskip
\begin{remark}\label{rem:nonintegralquantization}
    When $(M,Z)$ is a normal crossing divisor, consisting of immersed but not embedded hypersurfaces, we have no reason to believe that the quantization $Q(T_ZM,\omega)$ is an integer. So, relating the general case to a Fredholm index might not even be possible.
\end{remark}

\section{$b^k$-symplectic Manifolds}\label{sec:bkmanifold}

\sloppy Let $(M,Z,x_j)$ be an $n$-manifold and with a defined normal crossing divisor, i.e. $Z=Z_1\sqcup\cdots\sqcup Z_m$ where  $Z_j$ is an embedded submanifold of $M$ and is given by the defining functions $x_j$. For $k\geq 1$ define the subsheaf $\tau\subset \Gamma(TM)$ by

\[\tau:=\{X\in \Gamma(TM): dx_i(X)\ \textup{vanishes to order}\ k\ \textup{along}\ Z_i,\ 1\leq i\leq m\}.\]

Near a point in the multi-intersection such as $x_1=\cdots=x_i=0$, the sheaf is generated by the vector fields

\[x_1^k\frac{\partial}{\partial x_1},\cdots,x_i^k\frac{\partial}{\partial x_i},\frac{\partial}{\partial y_{i+1}},\cdots,\frac{\partial}{\partial y_n}\]
where $(x_1,\cdots,x_i,y_{i+1},\cdots,y_n)$ is a local coordinate on $M$. One can easily check that $\tau$ is closed under the Lie bracket and is locally free of constant rank $n$. Hence we obtain an E-tangent bundle, an algebroid, denoted by

\[\xymatrix{\bk T_ZM\ar[r]^{\rho}\ar[rd]&TM\ar[d]\\
{}& M }\]

\begin{remark}\label{rem:jet-dependence}
    For $k>1$, unfortunately, the definition of the sheaf depends on the $(k-1)$-jet class of the defining function (see \cite{scott2013geometry}). So we will fix the defining functions throughout. 
\end{remark}

\begin{lemma}\label{lem:stablyisomorphic}
    One has an isomorphism of vector bundles
    \[\bk T_ZM\oplus\bR\simeq TM\oplus\bR.\]
    The choice of this isomorphism is independent up to homotopy.
\end{lemma}

\begin{proof}
    The proof of this lemma is essentially the same as the proof of \cite[Theorem 2.8]{lin2022log}. One needs to change the $\rho\frac{df}{f}$  in \cite[Equation (5)]{lin2022log} to $\rho \frac{df}{f^k}$.
\end{proof}

\subsection{Liouville Volume}\label{ss:liouvillevolume} We will denote by $\bk\Omega_Z^r(M)$ the space of forms $\Gamma(\Lambda^r(\bk T_Z^*M))$.

Near an intersection of hypersurfaces like $x_1=\cdots=x_p=0$, one can write a Liouville-Laurent expansion for $\alpha\in\bk\Omega^r_Z(M)$ as 
\[\alpha=\beta+\sum_{1\leq i\leq p}\sum_{1\leq j\leq k}\frac{dx_i}{x_i^j}\alpha^1_{i,j}+\sum_{1\leq i_1<i_2\leq p}\sum_{1\leq j_1,j_2\leq k }\frac{dx_{i_1}}{x_{i_1}^{j_1}}\frac{dx_{i_2}}{x_{i_2}^{j_2}}\alpha_{I,J}^2+\cdots\]
Here, in the second sum,  $I=(i_1,i_2)$ and $J=(j_1,j_2)$. $\beta$ is a smooth $r$-form, and $\alpha^q_{I,J}$ is a smooth $(r-q)$-form.
Although this expression is not unique but the restricted forms $\alpha^q_{I,J}|_{x_1=\cdots=x_p=0}$ are unique. If $\alpha$ is a closed form, then one can take all $\alpha^q_{I,J}$ and $\beta$ to be closed.
Hence we obtain a \emph{residue} map
\begin{equation}\label{eq:residue}
    \res:\bk\Omega_Z(M)\to (\Omega^{r-1}(Z))^k\oplus \bigoplus_{i<j}(\Omega^{r-2}(Z_i\cap Z_j))^{k^2}\oplus\cdots
\end{equation}
\[\alpha\mapsto (\alpha^1_{i,j},\alpha_{I,J}^2\cdots).\]
We denote by $\res_{I,J}^q$ the component of the residue map that gives the coefficient $\alpha_{I,J}^p$.

When $\alpha$ is a top-degree form, similar to the log case, the integral of $\alpha$ can diverge due to the singularity of the form.  Let $M_{\epsilon}$ denotes the complement of $\bigcup_{i=1}^m x_i^{-1}(-\epsilon,\epsilon)$. One has the following:

\begin{proposition}
    There exists a polynomial $P_{\alpha}(t)$  of degree at most $n(k-1)$ for which
    \[\lim_{\epsilon\to 0}(P_{\alpha}(\frac{1}{\epsilon})-\int_{M_{\epsilon}}\alpha)=0.\]
    
\end{proposition}
\begin{proof}
The proof goes similarly to that of \cite[Theorem 5.3]{scott2013geometry} with some adjustments. This time we have to integrate terms of the form    
\[\frac{dx_{i_1}}{x_{i_1}^{j_1}}\cdots\frac{dx_{i_p}}{x_{i_p}^{j_p}}\alpha_{I,J}^p\] where $1\leq j_1,\cdots, j_p\leq k$, and $p\leq n$ (as there are at most $n$ intersections. After integrating, we obtain a term of the form 
\[\beta (\frac{1}{\epsilon})^{j_1+\cdots +j_p-p}.\]
So the largest power of $\frac{1}{\epsilon}$ is at most $nk-n=n(k-1)$.
\end{proof}

One defines the Liouville volume of $\alpha$ as the ``finite part'' of the integral above given by
\[\textup{fp}_Z\int_M\alpha :=P_{\alpha}(0).\]

\begin{remark}
    When $k=1$, the finite part of the integral coincides with the principal value integral given in Subsection \ref{ss:principalvalue}.
\end{remark}

\begin{example}\label{ex:sphereliouvillevolume}
    Let $M=S^2$ be the unit sphere in $\bR^3$ with the hypersurface $Z_a=z^{-1}(a)$, where $z$ is the projection onto the $z$-axis and $a\in (-1,1)$. For $k\geq 1$, define the $b^k$-form for the pair $(S^2,Z_a)$ by
    \[\omega_{a,k}:=\frac{dz\wedge d\theta}{2\pi(z-a)^k}.\]
    One has 
    \[\textup{fp}_{Z_{a}}\int_{S^2}\omega_{a,k}=\frac{1}{(k-1)!}\frac{d^{k-1}}{dx^{k-1}}|_{x=a}\big(\ln|\frac{1-x}{1+x}|\big)\]
  
\end{example}

\subsection{Algebroid cohomology of $\bk TM$}\label{ss:cohomologybktangentbundle}
First, one can define the ``smooth part'' map
\[H^r(\bk T_ZM)\to H^r(M)\] by the composition

\[\sm:H^r(\bk TM)\to (H^{n-r}(M))^*\buildrel{\textup{Poincar\'e}}\over\simeq H^r(M)\]
\[[\alpha]\mapsto [[\mu]\mapsto \textup{fp}_Z\int_M\alpha\wedge\mu].\]
Here we used a version of Poincar\'e duality which applies when $M$ is orientable. If $M$ is not orientable one can replace $(H^{n-r}(M))^*$ with the cohomology with coefficient in the orientation bundle $(H^{n-r}(M,\textup{or}_M))^*$ and its isomorphism with $H^r(M)$.

The algebroid cohomology of the $b^k$-tangent bundle has a similar expression to the cohomology of the log-tangent bundle (Proposition \ref{prop:cohomologylogtangent}):
\begin{theorem}
One has an isomorphism 
\[H^r(\bk T_ZM)\simeq H^r(M)\oplus \bigoplus_{i}(H^{r-1}(Z_i))^k\oplus \bigoplus_{i<j}(H^{r-2}(Z_i\cap Z_j))^{k^2}\oplus \cdots.\]    
\end{theorem}

\begin{proof}
    One can argue as \cite[Proposition 4.2]{scott2013geometry} and \cite[Proposition 28, Theorem 29]{miranda2021geometry} to obtains an exact sequence
    \[0\to\Omega^r(M)\to\bk\Omega^r_{ Z}(M)\buildrel{\res}\over\to  (\Omega^{r-1}(Z))^k\oplus \bigoplus_{i<j}(\Omega^{r-2}(Z_i\cap Z_j))^{k^2}\oplus\cdots\to 0.\]
    This exact sequence gives an exact sequence in the cohomology level:
    \[0\to H^r(M)\to H^r(\bk T_{Z}M)\to  \bigoplus_{i}(H^{r-1}(Z_i))^k\oplus \bigoplus_{i<j}(H^{r-2}(Z_i\cap Z_j))^{k^2}\oplus \cdots\to 0.\]
    This sequence splits by the smooth part map $
    \sm:H^r(\bk T_ZM)\to H^r(M)$. From this, the theorem follows.

\end{proof}


\subsection{Geometric quantization of $b^k$-symplectic manifolds}\label{ss:geomtricquantizationb^ksymplectic}
Let $\bk T_ZM$ be a $b^k$-tangent bundle associated with a defined normal crossing divisor $(M,Z,x_j)$. Assume $\omega$ is a symplectic form on $\bk T_ZM$ which is prequantizable, i.e., $\sm[\omega]$ is an integral cohomology class; so we can pick a line bundle $\L\to M$ with first Chern class $\sm[\omega]$.  Similar to Subsection \ref{ss:geometricquantizationlogsymlectic}, choosing a compatible complex structure on $\bk T_ZM$, one can define an algebroid-Dirac operator $D$. 

Again by Theorem \ref{prop:E-tangentintegrable}, there is a groupoid $r,s:\G\rightrightarrows M$ which integrates $\bk T_ZM$. Hence, by the prescription of Subsection \ref{ss:geometricquantizationmethod}, we obtain an index class $[D]\in K_0(\Psi^{-\infty}(\G))$.

Here we introduce a trace on $\Psi^{-\infty}(\G)$. Denote by $\Lambda$ the density bundle $|\det \bk T_Z^*M|\otimes \bC$ and define a vector bundle over $\G$ by 
\[\delta\Lambda^{1/2}:=r^*\Lambda^{1/2}\otimes s^*\Lambda^{1/2}.\]
Define the trace as the composition
\[\tr:\Gamma(\delta\Lambda^{1/2})\to \Gamma(\Lambda)\buildrel{\tau}\over\to\mathbb{C}\]
where $\tau$ is the Liouville integration of $b^k$-forms introduced as above. 

\begin{lemma}
    The map $\tr$ is a trace. 
\end{lemma}
The trace property of $\tr$ is equivalent to showing that $\tau(r_*f)=\tau(s_*f)$ where $f$ is a section of $\delta \Lambda=r^*\Lambda\otimes s^*\Lambda$ and $r_*,s_*:\Gamma(\delta \Lambda)\to \Gamma(\Lambda)$ are the pushforward maps given by integration along the range and source fibers
along the unit space. The proof of this lemma is very similar to the proof of \cite[Proposition 8.1]{sadegh2024fixed}, with small adjustments. 

Now one defines the geometric quantization by 
\[Q(\bk T_ZM,\omega)=\tr[D].\]
From \cite[Theorem 6.10]{sadegh2024fixed}, we obtain the following formula:
\[Q(\bk T_ZM,\omega)=\text{fp}_Z\int_M \Td^{\A}(\bk T_ZM)e^{\sm[\omega]}.\]

This proves  Theorem \ref{thm:b^kmanifoldquantizationformula}. 
\begin{example}\label{ex:prequantizedsphere}
    Consider the pair $(S^2,Z_a)$ as in Example \ref{ex:sphereliouvillevolume} with the defining function $z-a:S^2\to \bR$. We introduced the 2-form $\omega_k$ on the $\bk TS^2$ given by 
    \[\omega_{a,k}=\frac{dz\wedge d\theta}{2\pi(z-a)^k}.\]
    This is a symplectic form. So we deduce that $\omega_{a,k}$ is prequantizable if and only if 
    \begin{equation}\label{eq:prequantizedsphere}
        \frac{1}{(k-1)!}\frac{d^{k-1}}{dx^{k-1}}|_{x=a}\big(\ln|\frac{1-x}{1+x}|\big)\in\bZ.
    \end{equation}
\end{example}

\subsection{Hamiltonian action on $b^k$-symplectic manifolds.} \label{ss:bkHamiltonianquantization} Let $G$ be a compact connected Lie group with Lie algebra $\g$. 
Let $(\bk T_ZM,\omega)$ be a $b^k$-symplectic manifold. Assume $G$ acts on $M$ so that each connected component of $Z$, and all defining functions are kept invariant, and $\omega$ is preserved. For $X\in \g$, denote by $X_M$ (respectively, $X_Z$) the vector field of the infinitesimal action of $X$ on $M$ (respectively, Z).

\begin{definition}
    For each $1\leq j\leq k$, we define the $j$-th modular weight of the component $Z_i$ of $Z$, as the function $c_j:Z_i\to \g^*$ given by the equation
    \[\langle c_j,X\rangle=\iota_{X_Z}\res^1_{i,j}(\omega),\]
    where $\res_{i,j}^1$ is the component of the residue \eqref{eq:residue} that corresponds to $\alpha_{i,j}^1$. 
\end{definition}

\begin{remark}\label{rem:infinitesimalactionbk}
    Since the defining functions $x_i$ are $G$-invariant, for any $X\in\g$, $X_M.x_i=0$;. Hence $X_M\in \Gamma(\bk T_ZM)$ for every $k$.
\end{remark}

\begin{lemma}
    The modular weights are locally constant functions with values in the annihilator of the commutator, $[\g,\g]^0$.
\end{lemma}
\begin{proof}
    Since the residue function commutes with the de Rham differential and $\res_j(\omega)$ is an invariant form, using Cartan's formula, we obtain
    \[d\langle c_j,X\rangle=d\iota_{X_Z}\res_j(\omega)=-\iota_{X_Z}d\res_j(\omega)\]
    but $\res_j(\omega)$ is closed. So $c_j$ is locally constant. Due to the invariance of $c_j$ under the coadjoint action, it immediately follows that the weights are valued in $[\g,\g]^0$.
\end{proof}
\medskip
\begin{definition}
    Define $Z_{\neq0}$ to be the union of the components, at least one of whose modular weights is nonzero. For any component $W$ of $Z$, we define the \emph{modular degree} as 
    \[n_W=\max_i\{i:c_i|_W\neq0\}\]
    and define the \emph{leading modular weight} of $W$ as $c_W:=c_{n_W}$. 
    
    \medskip
     We say the $G$-Hamiltonian data $(M,Z,x_i,\omega,\mu)$ is of \emph{odd leading type} if for every component $W$ of $Z_{\neq 0}$, the modular degree $n_W$ is an odd integer.
\end{definition}

\medskip

\medskip
\begin{definition}
    We call the action of $G$ on the data $(M,Z,x_i,\omega)$ Hamiltonian if there exists an equivariant \emph{moment map} $\mu: M\backslash Z_{\neq0}\to \g^*$ such that
    \begin{enumerate}
        \item For any $X\in\g$
            \[\iota_{X_M}\omega=-d\langle \mu, X\rangle;\]
        \item The function
        \begin{equation}\label{eq:mumubar}
            \bar\mu:=\mu+\sum_{i=1}^m\sum_{j=2}^k\frac{c_j}{(j-1)x_i^{j-1}}- \sum_{i=1}^m c_1\ln|x_i|
        \end{equation}
            extends to a smooth map. 
    \end{enumerate}

    We call the Hamiltonian data $(M,Z,x_i,\omega,\mu)$ prequantizable if there exists an equivariant prequantum line bundle $\L\to M$, whose equivariant first Chern class $c_1^G(\L)$ equals $\sm[\omega]-\bar{\mu}\in H^2_G(M)$. 
    
\end{definition}

Since there is a $G$-action we may define the geometric quantization of a prequantized Hamiltonian $b^k$-symplectic manifold as an equivariant smooth function in $C^{\infty}(G)^G$, as follows. 

Let $\G\rightrightarrows M$ be an integration of the Lie algebroid $\bk T_ZM$. The action of the Lie group extends to an action on $\G$ via bisections. The algebroid Dirac operator $D$ constructed as before can be made $G$-equivariant (one chooses the complex structure on the algebroid equivariant, which is always possible). Then $D$ gives a family of honest Dirac operators on the source-fibers of $\G$ which is $G$-equivariant. This gives an index class in the equivariant K-theory $[D]\in K_0^G(\Psi^{-\infty}(\G))$. Pairing with the trace, gives the desired quantization:
    \[Q^G(\bk T_ZM,\omega):=\langle \tr,[D] \rangle\]
For the details of this construction, see \cite{sadegh2024fixed}. This quantization can be computed topologically (c.f. \cite[Theorem 6.10]{sadegh2024fixed}):
\begin{theorem}
For $g\in G$, we have
\[Q^G(\bk T_ZM,\omega,\mu)(g)=\textup{fp}_{Z^g}\int_{M^g}\frac{\Td^{\A}(\bk T_{Z^g}M^g)e^{\sm[\omega]-\bar{\mu}}}{\textup{ch}_g(\lambda_{-1}\nu^{0,1})}.\]
Here $\nu$ is the normal bundle of the fix point submanifold $M^g$ and $\lambda_{-1}\nu^{0,1}$ is the formal alternating sum $\sum (-1)^i\Lambda^i\nu^{0,1}$
\end{theorem}

\begin{example}\label{ex:spherehamiltonian}
    Continuing on the sphere example (Examples \ref{ex:sphereliouvillevolume} and \ref{ex:prequantizedsphere}), consider the $S^1$ action on the sphere by rotating about the $z$-axis with generating vector field $X_M=-2\pi\frac{\partial}{\partial \theta}$, then modular weights of $Z_a$ computes as 
    \[\langle c_j,X\rangle=\iota_{X_M}\res_j(\omega)=\begin{cases}
        -1&j=k\\
        0&j\neq k
    \end{cases}.\]
    
    When $k=1$, one has the moment maps $\mu_{1,b}(z,\theta)=-\ln(|z-a|)+b$ where $b$ is any real number. This case was studied in \cite{lin2022log}.

    Assume $k>1$, then one has the moment maps $\mu_{k,b}(z,\theta)=\frac{1}{(k-1)(z-a)^{k-1}}+b$ for $b\in\bR$. Assume $\sm[\omega]$ is integral, i.e., the equation \eqref{eq:prequantizedsphere} to hold. So we have the integer
    \[q:=\frac{1}{(k-1)!}\frac{d^{k-1}}{dx^{k-1}}|_{x=a}\big(\ln|\frac{1-x}{1+x}|\big)=-\frac{1}{k-1}(\frac{1}{(1-a)^{k-1}}-\frac{(-1)^{k-1}}{(1+a)^{k-1}}).\]
    For any two integers $n_1,n_2\in\bZ$ with $n_2-n_1=q$, one can find a unique $b$ with 
    \[\mu_{k,b}|_{z=1}=n_1, \ \ \ \mu_{k,b}|_{z=-1}=n_2.\]
    With this choice, the image of the pair $(\sm[\omega],\bar\mu)\in H^2_G(M)$ will be the equivariant first Chern class of an equivariant line bundle. Then the quantization formula above computes:
    \[Q^G(\bk T_{Z_a}S^2,\omega_{k,a},\mu_{k,b})(t)=\frac{t^{n_1}}{1-t}-\frac{t^{n_2}}{1-t}\] 
    for all $t\in S^1$.

\end{example}

\subsection{Geometric quantization as a Fredholm index}\label{ss:quantizationFredhol}
We will show the quantization above is an index of an honest Dirac operator; hence giving a positive answer to the conjecture in \cite[Remark 1.1]{guillemin2021geometric}.

By Lemma \ref{lem:stablyisomorphic}, we conclude
\begin{enumerate}

    \item All rational Pontryagin classes of $TM$ and $\bk T_ZM$ are the same.

    \item Since $TM\oplus \bC$ (i.e. $TM\oplus\bR^2$) and $\bk T_ZM\oplus\bC$ are isomorphic, and $\bk T_ZM$ is symplectic, $TM\oplus\bC$ carries an almost complex structure that is unique up to homotopy.
\end{enumerate}

Since $TM$ is stably almost complex, it carries a canonical $\spinc$ structure whose determinant line bundle is $\det_{\bC}(TM\oplus\bC)\simeq\det_{\bC}(\bk T_ZM)$.

Let $\slashed{D}^{\L}$ be the twisted Dirac operator on $\slashed{S}\otimes \L$. Then we have

$$\textup{Ind}(\slashed{D}^{\L})=\int_M\hat{A}(TM) \Ch(\textup{det}^{1/2}_{\bC}(TM\oplus\bC))\Ch(\L)$$
$$=\int_M \hat{A}(\bk T_ZM)\Ch(\textup{det}^{1/2}_{\bC}(\bk T_ZM))\Ch(\L).$$
$$=\int_M \hat{A}^{\A}(\bk T_ZM)\Ch(\textup{det}^{1/2}_{\bC}(\bk T_ZM))\Ch(\L).$$

\medskip
\begin{lemma}
    The algebroid Todd class 
$\Td^{\A}(\bk T_ZM)$ is represented by a smooth form, and 
\[[\Td^{\A}(\bk T_ZM)]=[\Td(TM\oplus\bC)].\]

\end{lemma}
\begin{proof}
In the second equality, we used the equality of the A-hat classes, due to the equality of the Pontryagin classes. The third equality follows from the following claim. We claim that the algebroid version of the A-hat class $\hat{A}^{\A}(\bk T_ZM)$ has the same algebroid cohomology class as that of $\hat{A}( TM)$ (hence $\hat{A}(\bk T_ZM)$). Note that the choice of the algebroid connection does not change the algebroid cohomology class $\hat{A}^{\A}(\bk T_ZM)$. So by choosing a connection on $TM$ and pulling it back to $\bk T_ZM$ via the anchor map, we may define the algebroid $\hat{A}^{\A}$ form. Since $\rho$ is an isomorphism, on an open and dense set we will have the equality $\hat{A}^{\A}(\bk T_ZM)=\hat{A}(TM)$, which extends to the entire $M$ by continuity. 

 Hence we have
\[\Td(TM\oplus\bC)=\hat{A}(TM) \Ch(\textup{det}^{1/2}_{\bC}(TM\oplus\bC))=\hat{A}(\bk T_ZM)\Ch(\textup{det}^{1/2}_{\bC}(\bk T_ZM))\]
\[=\hat{A}^{\A}(\bk T_ZM)\Ch(\textup{det}^{1/2}_{\bC}(\bk T_ZM))=\Td^{\A}(\bk T_ZM)\]

and thus the algebroid Todd class of $\bk T_ZM$ is actually smooth.
\end{proof}

Now, going to the quantization formula we obtain:

\[Q^G(\bk T_ZM,\omega,\mu)(g)=\textup{fp}_{Z^g}\int_{M^g}\frac{\Td^{\A}(\bk T_{Z^g}M^g)e^{\sm[\omega]-\bar{\mu}}}{\textup{ch}_g(\lambda_{-1}\nu^{0,1})}\]
\[=\int_{M^g}\frac{\Td(\bk T_{Z^g}M^g\oplus \bC)\Ch_g(L)}{\textup{ch}_g(\lambda_{-1}\nu^{0,1})}=\textup{Ind}(\slashed{D}^{\L})(g).\]

\subsection{Quantization commutes with reduction: $[Q,R]=0$}\label{ss:[Q,R]=0}

 Assume $(M,Z,x_i,\omega,\mu)$ is a $G$-Hamiltonian $b^k$-symplectic manifold, such that $0$ is a regular value for $\mu$ and $G$ acts freely on $\mu^{-1}(0)$. Thus $\mu^{-1}(0)$ and $M_0:=\mu^{-1}(0)/G$ are smooth manifolds and one has a principal $G$-bundle $p:\mu^{-1}(0)\to M_0$.

Similarly, for every component $Z_i$ of $Z$, considered as a submanifold of $M$, the quotient $\tilde Z_i:=(Z_i\cap \mu^{-1}(0))/G$ is smooth for which the function $x_i$ descend to a defining function $\tilde{x}_i$. Put $Z_0=\bigsqcup_i\tilde Z_i$.

\begin{lemma}\label{l:bk-symplecticreduction}
    The pair $(M_0,Z_0,\tilde{x}_i)$ is a defined normal crossing divisor. There is a unique symplectic form $\omega_0$ on the corresponding $b^k$-tangent bundle $\bk T_{Z_0}M_0$ with the property
    \[p^*\omega_0=\iota^*\omega\]
    where $\iota:\mu^{-1}(0)\to M$ is the inclusion.

    If $\L\to M$ is a $G$-equivariant prequantum line bundle, then the quotient $\L_0=\L|_{\mu^{-1}(0)}/G$ is a prequantum line bundle for $(\bk T_{Z_0}M_0,\omega_0)$.
\end{lemma}
\begin{proof}
Note that $\mu^{-1}(0)$ never meets the set $Z_{\neq0}$. Hence $\mu$ is smooth near $\mu^{-1}(0)$. For any stratum $N:=W_1\cap\cdots \cap W_p$ where $W_i$ is a component $Z\backslash Z_{\neq 0}$, $G$ acts freely on $\mu^{-1}(0)\cap N$, hence the restriction of $d\mu$ to the tangent space $T_mN$ is surjective for any $m\in N$; this is due to the relation $\langle d\mu,X\rangle=\iota_{X_M}\omega$ and nondegeneracy of $\omega$. Hence $\mu^{-1}(0)$ is transverse to every stratum $N$ of $Z\backslash Z_{\neq0}$. 

This shows that for any component $Z_i$ of $Z\backslash Z_{\neq0}$, $Z_i^{\mu}:=Z_i\cap \mu^{-1}(0)$ is a hypersurface in $\mu^{-1}(0)$. Since $ x^{\mu}_i:=x_i|_{\mu^{-1}(0)}$ is still a defining function for $Z_i^{\mu}$ in $\mu^{-1}(0)$. Since all the intersections are given by equations $x_{i_1}=\cdots =x_{i_r}=0$, the branches cross normally. 

Due to $G$-equivariance of $x_i^{\mu}$, they descends to a defining function for $\tilde{Z}_i:=Z_i^{\mu}/G$ in $M_0$ and by the same argument the hypersurfaces  all cross normally.  

Denote $\A:=\bk T_ZM$, define the subbundle $\A^{\mu}\to \mu^{-1}(0)$ as $\ker(d_{\A}\mu|_{\mu^{-1}(0)})$. This bundle has the rank $\dim(M)-\dim(G)$. The vector fields $X_M$, for $X\in\g$, gives subspace $\g_m\subset\A_m$. Since 
\[\omega_m(X_{M},Y_{M})=\pair{\d\mu_m(Y_M)}{X}=\pair{\mu(m)}{[Y,X]}=0
    \qquad(m\in P),\]
    hence $\g_m$ is an isotropic subspace of $\A_m$, and by definition $\A^{\mu}_m=\g_m^{\perp \omega}$. Hence, we obtain a symplectic vector bundle 
    $$\A_0:=\A^{\mu}\slash \g,$$ 
    by linear symplectic reduction. The symplectic form $\omega_0$ on $\A_0$ identifies by the property $p^*\omega_0=\iota^*\omega$. We still need to identify $\A_0$ with $T_{Z_0}M_0$.

    We have $\ker(\rho_m)\subset \A^{\mu}_m$ and $\ker(\rho_m)\cap\g_m=\{0\}$; so $\rho_m$ is injective on $\g_m$. Since $\d\mu_m(\rho(v))=\d_{\A}\mu|_m(v)=0$, we have $\rho(\A^{\mu}_m)\subset T_m\mu^{-1}(0)$. From these we conclude that $\rho$ descends to a well-defined map
    \[\rho_0:\A_0\to TM_0\]
    for which we have canonical isomorphism $\ker\rho_0\big|_{[m]}=\big(\ker\rho_m\oplus\g_m\big)/\g_m\simeq\ker\rho_m$.
    Similar to $\bk T_ZM$, the anchor $\rho_0$ is an isomorphism over $M_0\backslash Z_0$, and over a $p$-branch intersection has kernel of dimension $p$.

    Let $u\in \Gamma(\A^{\mu})$  be a $G$-invariant section (identified with a section of $\A_0$). For every $i$, 
    \[\d\tilde x_i\big(\rho_0(u)\big)\circ p=\d x_i\big(\rho(u)\big)\big|_P.\]
    Since $dx_i(\rho(u))$ vanishes to order $k$, we conclude $d\tilde{x}_i(\rho_0(u))$ vanishes to order $k$ and hence $\rho_0(\Gamma(\A_0))\subset\Gamma(T_{Z_0}M_0)$. The sheaf $\rho_0(\A_0)$ is locally free sheaf of the same rank as the sheaf $\Gamma(\bk T_{Z_0}M_0)$; since $ker(\rho_0)\simeq \ker(\rho_m)$, both sheaf have the same local generators near $Z_0$. Thus, we have the equality $\Gamma(\A_0)=\Gamma(\bk T_{Z_0}M_0)$; i.e. $\A_0\simeq \bk T_{Z_0}M_0$.
    Closedness of $\omega_0$ now follows from the equality $p^*\omega_0=\iota^*\omega$ and the fact that $p^*$ commutes from $d_{\A_0}$.

    We have the isomorphism $p^*:H^*(M_0)\simeq H^*_G(\mu^{-1}(0)) $ under which we have identifications
    \[c_1^G(\L)|_{\mu^{-1}(0)}=(\sm([\omega])-[\bar\mu])|_{\mu^{-1}(0)}=\sm([\omega])|_{\mu^{-1}(0)}\buildrel{(p^*)^{-1}}\over\longmapsto  \sm([\omega_0]).\]
    In the second equality we used the fact that near $\mu^{-1}(0)$ $\mu=\bar\mu$, and in the last identification we used the equality $p^*\omega_0=\iota^*\omega$. This shows that $\L_0$ is a prequantum line bundle for $(\bk T_{Z_0}M_0,\omega_0)$.
\end{proof}

\medskip
\begin{theorem}\label{thm:[Q,R]=0forbk-manifolds}
    Let $(M,Z,x_i,\omega,\mu)$ be a compact prequantizable $G$-Hamiltonian $b^k$-symplectic manifold, such that
    \begin{enumerate}
        \item  Assume  $0$ is regular value for $\mu$, and $G$ acts freely on $\mu^{-1}(0)$;

         \item For every stratum $N=W_1\cap\cdots\cap W_p$ where $W_1,\cdots, W_p$ are any distinct components of $Z_{\neq0}$, the non-negative linear combination of $\{c_{W_1},\cdots, c_{W_p}\}$ is strongly convex;

        \item $(M,Z,x_i,\omega,\mu)$ has odd leading type.
    \end{enumerate}

    If  $(M_0,Z_0,\omega_0)$ is the reduced data then the $b^k$-symplectic manifold $(T_{Z_0}M_0,\omega_0)$ is prequantizable and 
    \[Q(\bk T_{Z_0}M_0,\omega_0)=\big(Q^G(\bk T_ZM,\omega,\mu)\big)^G\]
    where $\big(Q^G(\bk T_ZM,\omega,\mu)\big)^G=\int_GQ^G(\bk T_ZM,\omega,\mu)(g)\ \d g$ is the ``multiplicity of the trivial representation.''
\end{theorem}


\begin{figure}[ht]
\centering
\begin{tikzpicture}[
    x=1cm, y=1cm, font=\footnotesize,
    sphere/.style   = {draw=black!70, line width=0.5pt},
    divisor/.style  = {draw=orange!85!red, line width=1.1pt},
    divisorb/.style = {draw=orange!85!red, line width=0.5pt, dashed},
    axisline/.style = {draw=black!75, line width=0.5pt, -{Stealth[length=4pt]}},
    upbar/.style    = {draw=blue!55!black, line width=2pt},
    lowbar/.style   = {draw=teal!70!black, line width=2pt},
    upar/.style     = {-{Stealth[length=3.5pt]}, blue!55!black, line width=0.7pt},
    lowar/.style    = {-{Stealth[length=3.5pt]}, teal!70!black, line width=0.7pt},
    wt/.style       = {circle, fill=orange!85!red, inner sep=1.25pt},
    onecov/.style   = {fill=blue!12},
    gapfill/.style  = {pattern=north east lines, pattern color=black!35},
    lbl/.style      = {inner sep=1pt}
  ]
 
 
\begin{scope}
  \draw[sphere] (0,0) circle (1.4);
  \draw[divisorb] (0,-1.066) ellipse (0.907 and 0.25);
  \draw[divisor]  (-0.907,-1.066)
        arc[start angle=180, end angle=360, x radius=0.907, y radius=0.25];
  \node[lbl,below left] at (-0.88,-1.20) {\color{orange!85!red}$Z_a$};
 
  \fill (0,1.4) circle (1.3pt);  \node[lbl,above] at (0,1.4) {$z=1$};
  \fill (0,-1.4) circle (1.3pt); \node[lbl,below] at (0,-1.4) {$z=-1$};
 
  \draw[upar]  (0.50,1.05) .. controls (0.95,0.35) and (0.92,-0.40) .. (0.72,-0.85);
  \draw[lowar] (-0.16,-1.39) .. controls (-0.55,-1.36) and (-0.80,-1.28) .. (-0.86,-1.18);
 
  \begin{scope}[shift={(2.9,0)}]
    \fill[onecov] (-0.20,-0.42) rectangle (0.20,0.42);
    \draw[axisline] (0,-1.80) -- (0,1.95);
    \node[lbl,above] at (0,1.95) {$\mathfrak{g}^*$};
    \foreach \l/\t in {-1/{n_1=-1}, 0/{0}, 1/{n_2=1}}
      { \draw (-0.07,0.42*\l) -- (0.07,0.42*\l);
        \node[lbl,left] at (-0.14,0.42*\l) {$\t$}; }
    \draw[upbar]  (0.34,-0.42) -- (0.34,1.72);
    \draw[upar]   (0.34,1.72) -- (0.34,1.95);
    \draw[lowbar] (0.70,0.42) -- (0.70,1.72);
    \draw[lowar]  (0.70,1.72) -- (0.70,1.95);
    \node[wt] at (0,-0.42) {}; \node[wt] at (0,0) {};
    \node[lbl,right,align=left] at (0.86,1.15) {};
    \node[lbl,right,align=left] at (0.86,-0.10) {};
  \end{scope}
 
  \node[lbl,align=center] at (1.55,-2.15)
      {$k=1$: both branches $\to+\infty$,\\ so the two images {overlap}};
  \node[lbl,align=center] at (1.55,-2.95)
      {$Q^{S^1}=t^{-1}+1$;\ at $\lambda=0$,\\ mult.\ $=+1=Q(\mathrm{pt})$};
\end{scope}
 
\usetikzlibrary{patterns}
\begin{scope}[shift={(7.5,0)}]
  \draw[sphere] (0,0) circle (1.4);
  \draw[divisorb] (0,0) ellipse (1.4 and 0.39);
  \draw[divisor]  (-1.4,0)
        arc[start angle=180, end angle=360, x radius=1.4, y radius=0.39];
  \node[lbl,left] at (-1.52,0.34) {\color{orange!85!red}$Z$};
 
  \fill (0,1.4) circle (1.3pt);  \node[lbl,above] at (0,1.4) {$z=1$};
  \fill (0,-1.4) circle (1.3pt); \node[lbl,below] at (0,-1.4) {$z=-1$};
 
  \draw[upar]  (0.45,1.12) .. controls (0.88,0.70) and (0.94,0.35) .. (0.90,0.14);
  \draw[lowar] (-0.45,-1.12) .. controls (-0.88,-0.70) and (-0.94,-0.35) .. (-0.90,-0.14);
 
  \begin{scope}[shift={(2.9,0)}]
    \fill[gapfill] (-0.20,-0.42) rectangle (0.20,0.42);
    \draw[black!45,line width=0.35pt] (-0.20,-0.42) rectangle (0.20,0.42);
    \draw[axisline] (0,-1.95) -- (0,1.95);
    \node[lbl,above] at (0,1.95) {$\mathfrak{g}^*$};
    \foreach \l/\t in {-1/{n_2=-1}, 0/{0}, 1/{n_1=1}}
      { \draw (-0.07,0.42*\l) -- (0.07,0.42*\l);
        \node[lbl,left] at (-0.14,0.42*\l) {$\t$}; }
    \draw[upbar]  (0.34,0.42) -- (0.34,1.72);
    \draw[upar]   (0.34,1.72) -- (0.34,1.95);
    \draw[lowbar] (0.34,-0.42) -- (0.34,-1.72);
    \draw[lowar]  (0.34,-1.72) -- (0.34,-1.95);
    \node[wt] at (0,-0.42) {}; \node[wt] at (0,0) {};
    \node[lbl,right,align=left] at (0.52,1.15) {};
    \node[lbl,right,align=left] at (0.52,-1.15) {};
    \node[lbl,right,align=left] at (0.52,0.02) { no level set};
  \end{scope}
 
  \node[lbl,align=center] at (1.55,-2.15)
      {$k=2$: branches $\to\pm\infty$,\\ so the two images are {disjoint}};
  \node[lbl,align=center] at (1.55,-2.95)
      {$Q^{S^1}=-1-t^{-1}$;\ at $\lambda=0$,\\ mult.\ $=-1\neq0=Q(\emptyset)$};
\end{scope}
 
\node[lbl,align=center,font=\scriptsize] at (4.5,-3.75)
  {\textcolor{blue!55!black}{\rule[0.4ex]{7pt}{1.6pt}}~image of the branch $z>a$
   \qquad
   \textcolor{teal!70!black}{\rule[0.4ex]{7pt}{1.6pt}}~image of the branch $z<a$
   \qquad
   \textcolor{orange!85!red}{$\bullet$}~weights of $Q^{S^1}$};
 
\end{tikzpicture}
 
\caption{The image of the moment map on $(S^2,Z_a)$ for
$\omega_{a,k}=\frac{\d z\wedge\d\theta}{2\pi(z-a)^k}$ with
$X_M=-2\pi\partial_\theta$, in the two parities. In both panels the weights
occurring in $Q^{S^1}$ are $\{-1,0\}$, and the shaded band on the
$\mathfrak{g}^*$-axis is the interval between $n_1$ and $n_2$ that carries
them. \emph{Left ($k=1$, odd leading type, $a=\frac{1-e^{2}}{1+e^{2}}$, $b=-1+\ln(1-a)$, $\mu=-\ln|z-a|+b$,
$n_1=-1$, $n_2=1$).} On the both sides of $Z_a$, the moment
map tends to $+\infty$ from either side and the branch images
$[n_1,\infty)$, $[n_2,\infty)$ overlap. On $[n_1,n_2)$ exactly one level set
lies over each value, giving multiplicity $+1$; above $n_2$ the two level sets
carry opposite orientations and cancel. \emph{Right ($k=2$, even leading type,
$a=b=0$, $\mu=1/z$, $n_1=1$, $n_2=-1$).} Now $(z-a)^{k-1}$ changes sign across
$Z$, the branches diverge to $+\infty$ and $-\infty$, and the images
$[n_1,\infty)$ and $(-\infty,n_2]$ are disjoint. The weights of $Q^{S^1}$ lie
in the hatched gap, over which there is no level set at all. At $\lambda=0$ the
multiplicity of the trivial representation is $-1$ whereas
$\mu^{-1}(0)=\emptyset$. This is Example~\ref{ex:QRfails}.}
\label{fig:parity}
\end{figure}
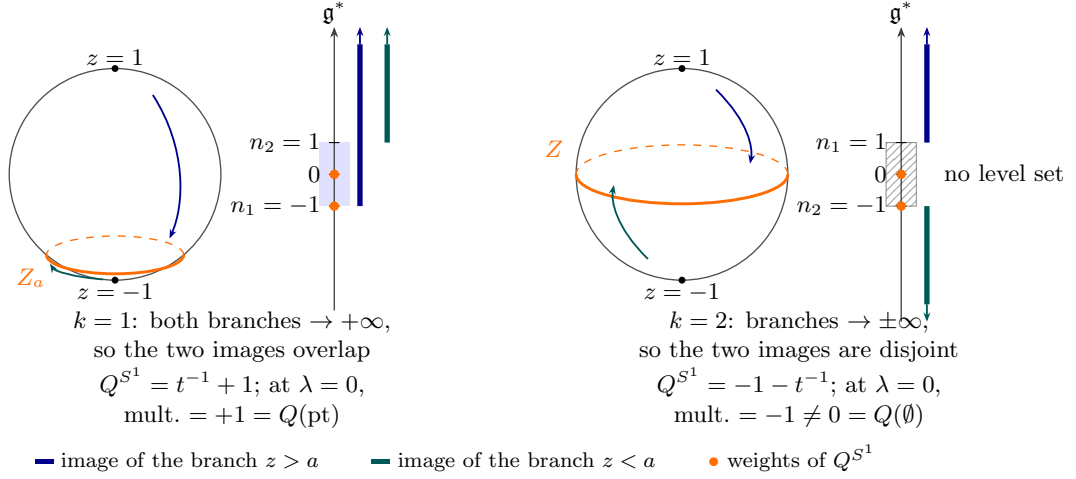


\begin{remark}

    If $Z$ has no crossings, hypothesis (2) is automatic. Indeed every stratum
  is then a single component $W\subset Z_{\neq0}$, and the cone
  $\{a\,c_W:a\ge0\}$ generated by one non-zero vector is
  strongly convex; equivalently $\mu$ is proper, as in
  \cite[Remark 4.2]{lin2022log}. In particular
  Theorem~\ref{thm:[Q,R]=0forbk-manifolds} applies with hypotheses (2), (3) alone to
  $b^k$-symplectic manifolds whose singular locus is a disjoint union of
  hypersurfaces, which is the setting of  \cite{guillemin2021geometric,guillemin2018convexity}.

   With crossings, hypothesis (2) cannot simply be deleted, because it is the
  right-hand side of Theorem~\ref{thm:[Q,R]=0forbk-manifolds} that fails to exist rather than the
  proof that fails. In \cite[Example 4.1]{lin2022log} one takes
  $M=S^2\times S^2$ with
  $\omega=\frac{\d z_1}{2\pi z_1}\wedge\d\theta_1
   -\frac{\d z_2}{2\pi z_2}\wedge\d\theta_2$
  and $Z_1=\{z_1=0\}$, $Z_2=\{z_2=0\}$ meeting along $S^1\times S^1$; the
  modular weights are $+1$ and $-1$, so the cone at that stratum is a line and
  $\mu^{-1}(0)$ is non-compact.  When $k=1$, Condition (2) is equivalent to $\mu$ being a proper map (see \cite[Lemma 4.8]{lin2022log}).
\end{remark}

\begin{remark}\label{rem:GMWcomparison}
  The restriction to odd leading type matches the assumption for
  formal quantization. In \cite{guillemin2021geometric} the formal
  quantization of a $b^m$-symplectic manifold with a Hamiltonian torus action
  of non-zero leading modular weight is shown to be a finite dimensional
  virtual $T$-module when $m$ is odd, whereas for $m$ even it is \emph{not}
  finite dimensional. Example~\ref{ex:QRfails} is an instance: the reduced
  spaces $\mu^{-1}(\lambda)/S^1$ are single points for every $|\lambda|\ge1$,
  so the formal quantization $\sum_\lambda\pm Q(\mu^{-1}(\lambda)/S^1)$ diverges. By contrast the
  quantization $Q^{S^1}(\bk T_{Z}S^2,\omega,\mu)=-1-t^{-1}$ of \eqref{eq:QisIndex} is a
  finite dimensional virtual representation. Thus for even leading type the
  index-theoretic quantization studied here is finite where the formal one is
  not, and the two necessarily disagree; for odd leading type
  Theorem~\ref{thm:[Q,R]=0forbk-manifolds} is what identifies them. (c.f. \cite{guillemin2019desingularizing})
\end{remark}

\subsection{A counterexample without odd leading type}\label{ss:counterexample}
 
The $[Q,R]=0$ statement genuinely fails without
Condition (3) of Theorem \ref{thm:[Q,R]=0forbk-manifolds}, so we record the smallest example.
 
\begin{example}\label{ex:QRfails}
  Following Examples \ref{ex:sphereliouvillevolume},\ref{ex:prequantizedsphere},and \ref{ex:spherehamiltonian}, take $k=2$, $M=S^2$, and $a=0$; $Z=\{z=0\}$, and
  \[\omega_{2,0}=\frac{\d z\wedge \d\theta}{2\pi z^{2}},\qquad
    X_M=-2\pi\frac{\partial}{\partial\theta},\qquad \mu_{2,0}=\frac1z .\]
  Here $n_{Z}=2$, $\langle c_{Z},X\rangle=-1$, and
  $q=-\tfrac{1}{1-a}-\tfrac{1}{1+a}\big|_{a=0}=-2\in\bZ$, so the data is
  prequantizable. The fixed point values are $\mu_{2,0}|_{z=1}=1$ and
  $\mu_{2,0}|_{z=-1}=-1$, so by the fixed point formula
  \[Q^{S^1}(\bk T_{Z}S^2,\omega_{2,0},\mu_{2,0})(t)=\frac{t^{1}-t^{-1}}{1-t}=-\,1-t^{-1},\]
  and the multiplicity of the trivial representation is $-1$. On the other
  hand $\mu_{2,0}=1/z$ omits the value $0$, so $\mu_{2,0}^{-1}(0)=\emptyset$, the
  reduced space is empty and its quantization is $0$. Thus
  \[\big[Q^{S^1}(\bk T_{Z}S^2,\omega_{2,0},\mu_{2,0})\big]^{S^1}=-1\neq 0=Q(\text{reduced space}).\]
  Note that $G$ acts freely on $\mu_{2,0}^{-1}(0)=\emptyset$, so no regularity
  hypothesis is being violated: it is the parity of $n_{Z}$ that fails.
\end{example}

\section{Proof of Theorem \ref{thm:[Q,R]=0forbk-manifolds}}\label{sec:[Q,R]proof}
The proof follows closely the proof of $[Q,R]=0$ in \cite{lin2022log}, with some adjustments.
\begin{proof}[\textbf{Proof of Theorem \ref{thm:[Q,R]=0forbk-manifolds}}]
    By Subsection~\ref{ss:geomtricquantizationb^ksymplectic} and the
$\spinc$ description of the quantization,
\begin{equation}\label{eq:QisIndex}
  Q^G(\bk T_ZM,\omega,\mu)=\textup{index}_G(\slashed D^{L}), 
\end{equation}
an honest equivariant index of an elliptic operator on the compact manifold
$M$. It therefore suffices to prove the corresponding statement for
$\slashed D^L$, and we may run the deformation argument of Paradan
\cite{paradan2001localization} and Paradan--Vergne
\cite{paradan2019witten} exactly as in \cite[Section 4]{lin2022log}.
We indicate only the points at which the $b^k$ setting differs.

\medskip
\noindent\emph{Step 1: a smooth regularization of $\mu$.}
Fix $\epsilon>0$. For $j\ge2$ let $p_{j,\epsilon}\in C^\infty(\bR)$ satisfy
$p_{j,\epsilon}(x)=x^{-(j-1)}$ for $|x|\ge2\epsilon$, let $p_{j,\epsilon}$
be constant on $|x|\le\epsilon$, and let
$|p_{j,\epsilon}|\le C_j\epsilon^{-(j-1)}$; let $\ln_\epsilon$ be as in
\cite[Definition 4.6]{lin2022log}. Replacing $x_i^{-(j-1)}$ by
$p_{j,\epsilon}(x_i)$ and $\ln|x_i|$ by $\ln_\epsilon|x_i|$ in
\eqref{eq:mumubar} defines $\widetilde\mu_\epsilon\in
C^\infty(M,\g^*)$, and the same substitution in the expression for the
smooth part $\sm(\omega)$ defines a smooth degenerate $2$-form
$\widetilde\omega_\epsilon$ with
$\iota_{X_M}\widetilde\omega_\epsilon=-\d\langle\widetilde\mu_\epsilon,X\rangle$.
Both agree with $\mu$, $\omega$ outside the $2\epsilon$-collar of
$Z_{\neq0}$.
 
Identify $\g\simeq\g^*$ by an invariant inner product and let
$\widetilde\kappa_\epsilon(m)=\big(\widetilde\mu_\epsilon(m)\big)_M(m)$ be
the Kirwan vector field, with vanishing locus
\[\widetilde\C=G\cdot\bigcup_{\beta\in\t^*_+}
   M^{\beta}\cap\widetilde\mu_\epsilon^{-1}(\beta).\]
 
\medskip
\noindent\emph{Step 2: the collar carries no zeros.}
This is the only step where hypotheses (2) and (3) are used, and the only
place where the $b^k$ estimates differ from the log case.

\begin{lemma}\label{lem:collarnozeros}
  Assume Conditions (2) and (3). Then $\mu$ is proper, and there are $\epsilon_0>0$ and
  a neighbourhood $U$ of $Z_{\neq0}$ such that for all
  $0<\epsilon<\epsilon_0$ the field $\widetilde\kappa_\epsilon$ does not
  vanish on $U$; consequently
  $\widetilde\C=\C:=G\cdot\bigcup_\beta M^\beta\cap\mu^{-1}(\beta)$ and
  $\mu=\widetilde\mu_\epsilon$ on $\C$.
  
\end{lemma}
\begin{proof}[\textbf{Proof of Lemma \ref{lem:collarnozeros}}]

    Fix a stratum $N=W_{1}\cap\dots\cap W_{p}$ and denote by $c_{W_l}$, $n_l$ and $x_l$ the leading modular weight, the modular degree and the defining function of $W_l$. The singular
  part of $\mu$ in \eqref{eq:mumubar} is a sum of functions of
  \emph{separate} variables,
  \begin{equation}\label{eq:separated}
    \mu=\ol\mu+\sum_{l=1}^{p}f_l(x_l),\qquad
    f_l(x_l)=-\sum_{j=2}^{k}\frac{c_{j}}{(j-1)x_l^{\,j-1}}+c_{1}\ln|x_l|,
  \end{equation}
  so the leading and lower-order terms may be compared one variable at a
  time. Put $t_l:=\big((n_{l}-1)x_l^{\,n_{l}-1}\big)^{-1}$ if $n_{l}\geq3$ and
  $t_l:=-\ln|x_l|$ if $n_{l}=1$; by Condition (3) we have $t_l>0$, and $t_l\to+\infty$ as
  $x_l\to0$. Then $f_l=-t_lc_{W_l}+r_l$, where $r_l$ collects the terms of order
  $j<n_{l}$. 

  Let $C_N$ be the cone of Condition (2). Since the leading weights are non-zero
  and $C_N$ is strongly convex, $\|\sum_l a_lc_{W_l}\|$ is positive on the
  compact simplex $\{a_l\ge0,\ \sum_la_l=1\}$, so by homogeneity there is
  $\delta>0$ with
  \begin{equation}\label{eq:conebound}
    \Big\|\sum_l t_lc_{W_l}\Big\|\ \ge\ \delta\sum_l t_l
    \qquad\text{for all }t_l\ge0 .
  \end{equation}
  Combining \eqref{eq:separated}--\eqref{eq:conebound},
  \[\|\mu\|\ \ge\ \delta\sum_l t_l-\|\ol\mu\|_\infty-\sum_l|r_l|
    \ =\ \delta\sum_l t_l+\ \textup{lower-order terms},\]
  which tends to $\infty$ as any $x_l\to0$. As $M$ is compact, this proves
  that $\mu$ is proper.

      Now, \cite[Lemma 4.9]{lin2022log} applies unchanged: if
  $X=\sum_l t_l c_{W_l}\neq0$ then
  $u_l:=\iota(X_{Z_{l}})\res^1_{l,n_l}(\omega)
  =\sum_{l'}t_{l'}\langle c_{W_l},c_{W_{l'}}\rangle$
  satisfies $\sum_l t_l u_l=\|X\|^2>0$, so some $u_l\neq0$ and $X_M$ does not
  vanish on $N$. Since non-vanishing is open and invariant under positive
  scaling, we obtain a scale-invariant open
  $\U_N\supset\textup{span}\{c_{W_l}\}\setminus\{0\}$ and a
  neighbourhood $U_N$ of $N$ on which $X_M\neq0$ for all $X\in\U_N$.
 
  Finally, the truncation obeys the same bounds with $|x_l|$ replaced by
  $\epsilon$.
  Hence on $U_N$ and
  sufficiently near $N$ the map $\widetilde\mu_\epsilon$ is a perturbation,
  small relative to $\sum_l t_l$, of the non-negative combination
  $\sum_l t_lc_{W_l}$ of the leading weights, which by \eqref{eq:conebound} is
  non-zero of norm at least $\delta\sum_l t_l$. So $\widetilde\mu_\epsilon$
  takes values in $\U_N$ near $N$ and $\widetilde\kappa_\epsilon\neq0$ there.
  There are finitely many strata, so a single $\epsilon_0$ works. The last
  two assertions follow as in \cite[Corollary 4.11]{lin2022log}.
  \end{proof}

    \medskip
\noindent\emph{Step 3: non-abelian localization.}
With $\widetilde\kappa:=\widetilde\kappa_\epsilon$ smooth on $M$ and
vanishing exactly on $\C$, the deformed symbol
$\widetilde\sigma(\xi)=c(\xi-\widetilde\kappa)$ is homotopic to the symbol
of $\slashed D^L$ and is transversally elliptic. Atiyah's excision property
gives, in $R^{-\infty}(G)$,
\begin{equation}\label{eq:localization}
  \textup{index}_G(\slashed D^{L})
  =\sum_{\beta\in\B}\textup{index}_G\big([\widetilde\sigma_\beta]\big),
\end{equation}
where $\B\subset\t^*_+$ is the finite set of $\beta$ with
$M^\beta\cap\mu^{-1}(\beta)\neq\emptyset$. This is \cite[Section 4.2]{lin2022log} verbatim; nothing about the order of the poles
enters.

\medskip
\noindent\emph{Step 4: only $\beta=0$ contributes to the invariant part.}
Let $0\neq\beta\in\B$, let $M_\beta\subset M^\beta$ be the union of
components meeting $\mu^{-1}(\beta)$, and let $\nu$ be its normal bundle.
The criterion of \cite[Theorem 9.6]{paradan2019witten} reduces the
claim to positivity of the eigenvalue of $b=-\sqrt{-1}\beta$ on
\[\L'_\beta=\det(\nu,J_\beta)^{1,0}\otimes
  \big(\det(\g/\g_\beta,j_\beta)^{0,1}\big)^{2}\otimes
  \det(\nu,J)^{1,0}\otimes L^{2}\big|_{M_\beta}.\]
Two points needs to be made here. First, the
anti-canonical line bundle of the spinor module is
$\det_{\bC}(\bk T_ZM)$, since $S$ arises from a complex structure on
$\underline{\bR}^2\oplus TM\simeq\underline{\bR}^2\oplus\bk T_ZM$; and since
$G$ preserves each $Z_i$, every $Z_i$ is transverse to $M_\beta$, whence
$\bk T_ZM|_{M_\beta}\simeq\bk T_{Z\cap M_\beta}M_\beta\oplus\nu$ and the
factor on which $\beta$ acts trivially may be discarded. Second, by
Lemma~\ref{lem:collarnozeros} the set $\C_\beta$ lies outside the collar,
where $\widetilde\mu_\epsilon=\mu$ and $\omega$ is a genuine fibrewise
symplectic form on $\bk T_ZM$; so for $X\in\g/\g_\beta$ the computation
\[\omega_m\big(X_M,(\ad_\beta X)_M\big)
 =-\d_m\langle\mu,X\rangle\big((\ad_\beta X)_M\big)
 =-\langle\ad^2_\beta X,X\rangle=\|\ad_\beta X\|^2>0\]
goes through unchanged, and Kostant's formula gives eigenvalue
$2\pi\|\beta\|^2>0$ on $L_m$. The remaining bookkeeping of eigenvalue signs
is \cite[Lemma 4.17]{lin2022log} word for word. Hence taking $G$-invariants
in \eqref{eq:localization} leaves only the $\beta=0$ term.

\medskip
\noindent\emph{Step 5: the $\beta=0$ term is $Q(M_0)$.}
Since $G$ acts freely on $\mu^{-1}(0)$ we have
$TM|_{\mu^{-1}(0)}\simeq p^*TM_0\oplus\underline\g\oplus\underline{\g^*}$
with $p\colon\mu^{-1}(0)\to M_0$ the quotient map, and
$\underline\g\oplus\underline{\g^*}\simeq\underline{\g_\bC}$ carries its
canonical complex structure. Set
\[S_0':=\Hom_{\Cl(\g_\bC)}\big(\textstyle\bigwedge\underline{\g_\bC},\,
   S|_{\mu^{-1}(0)}\big)\big/G .\]
By \cite[Theorem 9.6]{paradan2019witten},
$\textup{index}_G([\widetilde\sigma_0])^{G}$ is the index of the Dirac
operator on $M_0$ acting on $S_0'\otimes L_0$. It remains to identify $S_0'$
with the spinor module $S_0$ determined by the $b^k$-symplectic structure
$(\bk T_{Z_0}M_0,\omega_0)$ and the induced orientation, and for this one
uses the equivariant normal form of
Lemma~\ref{lem:bknormalform} below. On the model neighborhood
$U\simeq\mu^{-1}(0)\times U_{\g^*}$ one has
$\bk T_{Z\cap U}U\simeq (p\circ\pr_1)^*\,\bk T_{Z_0}M_0
 \oplus\underline{\g_\bC}$, and the stated form of $\omega$ shows that
$(p\circ\pr_1)^*J_0\oplus J_{\g_\bC}$ is $\omega$-compatible. As the space of
compatible complex structures is contractible, $J$ may be taken of this form,
whence $S|_U\simeq (p\circ\pr_1)^*S_0\otimes\bigwedge\underline{\g_\bC}$ and
$S_0'\simeq S_0$.

\end{proof}

\medskip 
\begin{lemma}[$b^k$ equivariant normal form]\label{lem:bknormalform}
  Under the hypotheses of Theorem~\ref{thm:[Q,R]=0forbk-manifolds} there is a $G$-equivariant
  diffeomorphism $\varphi\colon\mu^{-1}(0)\times U_{\g^*}\to U$ onto a
  $G$-invariant neighbourhood $U$ of $\mu^{-1}(0)$ in $M$, where
  $U_{\g^*}\ni0$ is $G$-invariant, such that
  \begin{enumerate}
    \item $\varphi^{-1}(W\cap U)=\big(W\cap\mu^{-1}(0)\big)\times
      U_{\g^*}$ for every component $W\subset Z\backslash Z_{\neq0}$ compatible with the defining function of $W$ up to order $k-1$;
    \item $\varphi^*\omega=\pr_1^*p^*\omega_0
      +\d\langle\pr_2,\pr_1^*\theta\rangle$ for a connection $\theta$ on the
      principal bundle $p\colon\mu^{-1}(0)\to M_0$.
  \end{enumerate}
\end{lemma}
 The Lemma above is the consequence of the of Lemma \ref{l:bk-symplecticreduction} and \cite[Theorem 4.4.1]{lin2023symplectic}.

\bibliographystyle{amsplain}

\providecommand{\bysame}{\leavevmode\hbox to3em{\hrulefill}\thinspace}
\providecommand{\MR}{\relax\ifhmode\unskip\space\fi MR }
\providecommand{\MRhref}[2]{%
  \href{http://www.ams.org/mathscinet-getitem?mr=#1}{#2}
}
\providecommand{\href}[2]{#2}


\bibliography{references}

\end{document}